\documentclass[a4paper,slashbox, oneside,11pt]{article}
 \usepackage{graphicx}
 \usepackage{epsfig}
\usepackage{amsmath}
\usepackage{amsthm}
\usepackage{amssymb}
\usepackage{epic,eepic}
\usepackage{cite}

\newtheorem{lem}{Lemma}

\newtheorem{prop}{Proposition}
\newcommand{\R}{{ R\hspace*{-1.5ex}\rule{0.15ex}{1.5ex}\hspace*{0.9ex}}}
\newfam\msbmfam
\font\tenmsbm=msbm10 \textfont\msbmfam=\tenmsbm
\font\sevenmsbm=msbm7 \scriptfont\msbmfam=\sevenmsbm
\font\fivemsbm=msbm5 \scriptscriptfont\msbmfam=\fivemsbm
\def\msbm {\fam\msbmfam\tenmsbm}
\def \R {{\msbm R }} 
\def \N {{\msbm N}} 
\usepackage{authblk}
\begin{document}
\title{ On a regularization approach for solving the inverse Cauchy Stokes problem}
\author[a]{ A. Chakib}
\author[b]{ A. Nachaoui}
\author[a]{ M. Nachaoui}
\author[a]{H. Ouaissa\footnote{corresponding author: hamid04gmi@gmail.com}}
\affil[a]{Laboratoire de Math\'ematiques et Applications
Universit\'e Sultan Moulay Slimane, Facult\'e des Sciences et
Techniques, B.P.523, B\'eni-Mellal, Maroc.} \affil[b]{Laboratoire de
Math\'ematiques Jean Leray UMR6629 CNRS / Universit\'e de Nantes  2
rue de la Houssini\`ere, BP92208 44322 Nantes, France. }
\date{}
\maketitle
\begin{abstract}
In this paper, we are interested to an inverse Cauchy problem governed by Stokes equation, called the data completion problem. It consists in determining the unspecified fluid velocity, or one of its components over a part of its boundary, by introducing given measurements on its remaining part. As it's known,  this problem is one of highly ill-posed problem in the  Hadamard's sense \cite{had}, it is then an interesting challenge to carry out a numerical procedure for approximating their solutions, mostly in the particular case of noisy data. To solve this problem, we propose here a regularizing approach based on a coupled complex boundary method, originally proposed in \cite{source}, for solving an inverse source problem. We show the existence of the regularization optimization problem and  prove the convergence of subsequence of optimal solutions of Tikhonov regularization formulations  to the solution of the Cauchy problem. Then we suggest the numerical approximation of this problem using the adjoint gradient technic and the finite elements method of $P1-bubble/P1$ type. Finally, we provide some numerical results showing the accuracy, the effectiveness and robustness of the proposed approach.
\end{abstract}
\section{Introduction}
This paper reveals a computational method for solving the inverse Cauchy problem for Stokes equation. The Cauchy problem consists in recovering a missing boundary condition, usually contaminated with noise, from partial measurements of Dirichlet and Neumann conditions on an accessible part of the boundary. It is well know that the inverse Cauchy problem is ill-posed in the Hadmard's sense\cite{had}. This means that a small perturbation in given data may result to an arbitrary large errors in solution. This  aspect is especially important from a practical point of view, since in the realistic situations, boundary data come from measurements which contain noise. It is then an interesting challenge to carry out stable numerical approaches, mostly in the particular case of noisy data. More precisely, regularization methods should be employed, to obtain regularized stable and accurate numerical solution.

Over the last two decades a large body of literature were dedicated to seek for a stable and efficient methods which deals with the ill-posed behaviors of the Cauchy problem. So many research work have been conducted to propose different regularization approaches, we can not list all of them. The most popular one is Tikhonov's regularization method \cite{Tik,reg1,calvtikhonov}, which transforms the original ill-posed into a well posed problem by minimizing the $L^2$-norm of the solution subjected to the constraint equation. Then other methods have been proposed to regularize the Cauchy problem,
 we can mention for instance the alternating method \cite{hol,kozlov}, the universal method \cite{carl}, the quasi-reversibility method \cite{lionsquas,stab,qumeth}, the technic of fundamental solution \cite{fairmethod} and improved nonlocal boundary value problem method \cite{nl}, etc. Nevertheless, the literature devoted to the Cauchy problem for linear elliptic equations is very rich (see for example \cite{stab2,SR,CN, stabestim,haonon,regCauch} and the references therein).

 In this paper, we deal with a regularization method for solving an inverse Cauchy problem governed by Stokes equation, which consists in determining the unspecified fluid velocity, or one of its components over a part of its boundary, by introducing given measurements on its remaining part. The severe ill-posedness of this inverse Stokes problem lies in the fact that the solution's behavior hardly changes when there is slight change in the data. To overcome this server instability, we suggest a regularization coupled complex boundary method, originally proposed in \cite{source}, for solving an inverse source problem. This approach consists to include all data on the boundary, the known and unknown ones, in a complex Robin boundary on the whole boundary. the Cauchy problem is transferred then  into a complex Robin boundary problem of finding the unknown data such that the imaginary part of the solution equals zero in the domain. Then an optimization formulation of the problem is proposed and the Tikhonov regularization approach is performed to
resulting  optimization problem. Some theoretical analysis results on the coupled complex boundary method combined with Tikhonov regularization approach are given. More precisely, we show the existence of the regularization optimization problem, we identify it with respect to the solution of the adjoint state problem and  prove the convergence of subsequence of optimal solutions of Tikhonov regularization formulations  to the solution of the Cauchy problem.  Moreover, using the adjoint gradient technic, a simple solver is proposed to compute
the regularized solution. Thus, no iteration is needed and the resolution is fast. The finite-element method of $P_{1Bubble}/P_1$ type's \cite{finite} is used for the discretization. Numerical results are given to confirm that the proposed approach produces convergent and stable numerical solutions with respect to decreasing the amount of noise added into the input data.

This paper is organized as follows, in the first section, we begin by giving the setting of the inverse Cauchy problem and present its reformulation as an equivalent coupled complex boundary value  Cauchy problem of Robin condition, with null imaginary solution.  In section 3,  we  propose an optimization formulation of the coupled complex boundary value problem, then we suggest its regularization using the Tikhonov regularization framework for the Cauchy problem with noisy data. We show the existence of the resulting optimization problem, we identify it with respect to the solution of the adjoint state problem and  we prove the convergence of subsequence of optimal solutions of Tikhonov regularization formulations  to the solution of the Cauchy problem. In section 4, we present an algorithm for solving the regularized optimal solution of the optimization problem based on adjoint gradient technic, we propose then the approximation of the optimization problem using finite-element method $P_{1Bubble}/P_1$. Finally, we present some numerical results showing the effectiveness and the feasibility of the proposed approach.
\section{Setting and formulation of the problem}

Let $\Omega\subset\R^d$ (d=2,3) be an open bounded domain with Lipschitz boundary  $\Gamma:=\partial\Omega=\Gamma_1\cup\Gamma_0$, where $\Gamma_1\cap\Gamma_0=\emptyset $. Denote by $n$ the unit outward normal to $\Gamma$. For given functions $\mu$ $f$ defined in $\Omega$ and given Cauchy data $\psi$ and $\kappa$ defined $\Gamma_0$, we consider the following Cauchy problem governed by the Stokes equation, which consists in finding $(\varphi,\zeta)$ defined on $\Gamma_1$, solution of
\begin{equation}
\label{eq1}
\left\lbrace
\begin{array}{llll}
-2\mu\,div(D(u))+\nabla p=f&in  &\Omega\\
div(u)=0 &on &\Omega\\
\sigma(u)n=\psi,\qquad u=\kappa & on & \Gamma_0\\
 \sigma(u)n=\varphi,\qquad u=\zeta & on &\Gamma_1\\
\end{array}\right.
\end{equation}
where $D(u)$ is the deformation tensor given by
$$D(u)=\frac{1}{2}(\nabla u+\nabla^t u)$$
and $\sigma(u)$ is the Cauchy  tensor given by
$$\sigma(u)=2\mu D(u)-p\,I,\quad I \mbox{ is the identity matrix},$$
The vector $u$ represents the velocity field of the fluid, the scalar function $p$ designates the associated pressure, $f$ being the force field acting on the system and $\mu>0$ is the kinematic viscosity coefficient.

   We note that the part of boundary $\Gamma_0$ is over-determined, since two boundary conditions of Dirichlet and Neumann type are imposed on it, while $\Gamma_1$ is the non-accessible part of the boundary, on which a missing boundary condition must be recovered. To do that, let us first give some useful notations and definitions. We denote by
   $$\Theta:=L^2(\Omega)^d,$$
   $$\Theta_0=\{v\in L^2(\Omega)\,\,\,/\,\,\,\int_{\Omega}v(x)\,dx=0\},$$
   $$\Theta_{\Gamma_0} := L^2(\Gamma_0)^d $$
   and
   $$\Theta_{\Gamma_1} := L^2(\Gamma_1)^d.$$
   $H^m(\Omega)^d$ (for $m\in \N$) will denote the Sobolev complex space equipped with the inner product
   $((\cdot,\cdot))_{m,\Omega,d}$ and the norm $||| \cdot |||_{m,\Omega}$ defined respectively  as follows:
   $$\displaystyle \forall u, v \in H^m(\Omega)^d \quad((u, v))_{m,\Omega,d} = \sum_{j=1}^{d}(u_j,\bar{v_j})_{m,\Omega}$$
   $$|||v|||_{m,\Omega,d} =\sqrt{((v, v))_{m,\Omega,d}}.$$
   We denote in particular by $V = H^1(\Omega)^d$ and we will use the following assumption
   $$f\in \Theta, \quad \psi\in \Theta_{\Gamma_0} \quad \mbox{and} \quad \kappa \in \Theta_{\Gamma_0}.$$
   In the following, we will denote by  $c$  a generic positif constant, which may have a different value at a different place.

In the sequel, we will propose a regularization method allowing us to obtain a stable approximate solution of the ill-posed Stokes inverse problem (\ref{eq1}). This is based on a coupled complex boundary method, originally proposed in \cite{source}, for solving an inverse source problem. For this, we will reformulate our inverse problem into a complex Cauchy problem.

\subsection{A complex formulation of the Cauchy problem}
The proposed formulation consists in combining the two given boundary conditions on $\Gamma_0$, to obtain a complex one of Robin type. We can then consider the complex boundary value problem: find $(\varphi,\zeta)$ defined on $\Gamma_1$, solution of
\begin{equation}\left\lbrace
\begin{array}{llll}
-2\mu\,div(D(u))+\nabla p=f&in  &\Omega\\
div(u)=0 &on &\Omega\\
\sigma(u)n+iu=\psi+i\kappa & on & \Gamma_0\\
 \sigma(u)n+iu=\varphi+i\zeta & on &\Gamma_1,\\
\end{array}\right.
\label{pb5}
\end{equation}
where $i$ is the imaginary  number.

It is clear that if $(\varphi,\zeta)$ is solution of $(\ref{eq1})$ then it is solution of $\ref{pb5}$.
Conversely, let  $(\varphi,\zeta)$ be a solution of  $\ref{pb5}$, then the associate solution $(u,p)$ can be written  $u=u_1+iu_2$ and $p=p_1+ip_2$, where $(u_1,p_1)$ and $(u_2,p_2)$  are respectively the real and imaginary parts of $(u,p)$, which are respectively solution of
\begin{equation}\left\lbrace
\begin{array}{llll}
-2\mu\,div(D(u_1))+\nabla p_1=f&in  &\Omega\\
div(u_1)=0 &on &\Omega\\
\sigma(u_1)n-u_2=\psi & on & \Gamma_0\\
 \sigma(u_1)n-u_2=\varphi & on &\Gamma_1\\
\end{array}\right.
\label{pb_6}
\end{equation}
\begin{equation}\left\lbrace
\begin{array}{llll}
-2\mu\,div(D(u_2))+\nabla p_2=0&in  &\Omega\\
div(u_2)=0 &on &\Omega\\
\sigma(u_2)n+u_1=\kappa & on & \Gamma_0\\
 \sigma(u_2)n+u_1=\zeta & on &\Gamma_1\\
\end{array}\right.
\label{pb7}
\end{equation}
So if $u_2=0$ and $p_2=0$ in $\Omega$, then  $u_2=0$ and $\sigma(u_2)n=0$ on $\Gamma$. Thus  from the boundary value problems (\ref{pb_6})-(\ref{pb7}), we get that $(\varphi,\zeta)$ is solution of $(\ref{eq1})$ associated to $(u_1,p_1)$.
We can then reformulate the Cauchy problem (\ref{eq1}) as follows : find $(\varphi,\zeta)$ defined on $\Gamma_1$ such that
\begin{equation}\left\lbrace
\begin{array}{l}
u_2=0,\,\, p_2=0 \,\,in \,\, \Omega\\
\mbox{where}\,\, (u_2,p_2) \mbox{ is the imaginary part of  } (u,p) \mbox{ the solution of the boundary value problem } (\ref{pb5}),\\
\end{array}\right.
\label{pb7-1}
\end{equation}
where  the weak formulation of (\ref{pb5})  is given by
\begin{equation}
\label{weak10}
\left\lbrace
\begin{array}{ll}
\mbox{find } (u,p)\in V\times \Theta_0&\\
a(u,v)+b(v,p)=F(\varphi,\zeta,v)& \forall v\in V \\
b(u,q)=0&\forall q\in \Theta_0
\end{array}
\right.
\end{equation}
and the forms $a$, $b$ and $F$ are defined respectively  as follows
$$a(u,v)=2\mu\int_{\Omega}D(u):D(\bar{v})dx+i\int_{\Omega}u\bar{v}ds\qquad \forall u,v\in V,$$
$$b(u,q)=-\int_{\Omega}q \,div(u)dx\qquad \forall (u,q)\in V\times \Theta_0$$
and
$$F(\varphi,\zeta,v)=\int_{\Omega}f\bar{v}dx+\int_{\Gamma_0}(\psi+i\kappa)\bar{v}ds+\int_{\Gamma_1}(\varphi+i\zeta)\bar{v}ds\qquad\forall v\in V$$
In order to solve the problem $(\ref{pb7-1})$, we will use a minimization formulation based on the Tikhonov regularization. This requires to show the well posedness of the  variational formulation (\ref{weak10}) of the boundary value problem (\ref{pb5}), for all given couple $(\varphi,\zeta)$. This result is stated in the following proposition.
\begin{prop}\label{pro3.1}
For given $f\in \Theta,$ $(\psi,\kappa)\in \Theta_{\Gamma_0}\times \Theta_{\Gamma_0}$ and $(\varphi,\zeta)\in \Theta_{\Gamma_1}\times \Theta_{\Gamma_1}$, the problem (\ref{weak10}) admits a unique solution $(u,p)\in V\times \Theta_0)$ which depends continuously on data. There exists then two constant $\alpha>0$  and $\beta>0$, such that
\begin{equation}
\label{11}
|||u|||_{1,\Omega}\leq \frac{c}{\alpha}(|||f|||_{0,\Omega}+|||\psi|||_{0,\Gamma_0}+|||\kappa|||_{0,\Gamma_0}+|||\varphi|||_{0,\Gamma_1}+|||\zeta|||_{0,\Gamma_0})
\end{equation}
 and
\begin{equation}
\label{11.5}
||p||_{0,\Omega}\leq \frac{c}{\beta}(|||f|||_{0,\Omega}+|||\psi|||_{0,\Gamma_0}+|||\kappa|||_{0,\Gamma_0}+|||\varphi|||_{0,\Gamma_1}+|||\zeta|||_{0,\Gamma_0})
\end{equation}
where $c$ designates the generic constant.
\end{prop}
\begin{proof}
The proof of the existence and uniqueness of the solution of (\ref{weak10}), follows from the continuity of $a$ on $V\times V$, its coercivity on $V$, the continuity of $F$ on $V$ and the inf-sup condition on $b$  (see for example \cite{infsup}). Indeed
For any $u,v\in V,$ by using the H\"older inequality and the continuity of the trace  operator from $V$ to $\Theta_{\Gamma_0}$ and $\Theta_{\Gamma_1}$, we get the continuity of $a(\cdot,\cdot)$ and $F(\varphi,\zeta,\cdot)$:
\begin{equation}
|a(u,v)|\leq c|||u|||_{1,\Omega}|||v|||_{1,\Omega}
\label{12}
\end{equation}
\begin{equation}
\label{13}
|F(\varphi,\zeta,v)|\leq c(|||f|||_{0,\Omega}+|||\psi|||_{0,\Gamma_0}+|||\kappa|||_{0,\Gamma_0}+|||\varphi|||_{0,\Gamma_1}+|||\zeta|||_{0,\Gamma_1})|||v|||_{1,\Omega}
\end{equation}
Moreover, the coercivity of $a$ is obtained due to the Korn's inequality, there exists then $\alpha >0$, such that:
\begin{equation}
\label{14}
Re(a(u,u))=2\mu|||D(u)|||^2_{0,\Omega}\geq c |||u|||^2_{1,\Omega},  \quad \forall u\in V
\end{equation}
To conclude, using the same technics as in \cite{infsup}, we prove the inf-sup condition on $b$, which means that there exists a constant $\beta >0$
\begin{equation}
\label{15.1}
\inf_{q\in \Theta,q\neq 0} \sup_{v\in V,v\neq 0}\frac{b(v,q)}{|||v|||_{1,\Omega}\|q\|_{L^2(\Omega)}} \geq \beta
\end{equation}
Therefore, the problem (\ref{weak10}) admits a unique solution $(u,p)\in V\times \Theta_0.$

The estimations (\ref{11}) and (\ref{11.5})  follow then directly from  the use of the inequalities (\ref{13}),(\ref{14}) and (\ref{15.1}).
\end{proof}

\section{Tikhonov regularization approach and theoretical analysis}
In this section, we will present a minimization approach based on Tikhonov regularization  for solving the problem $(\ref{pb7-1})$. Then we show the existence of the optimal solution  and  prove the convergence of subsequence of optimal solution of Tikhonov regularization formulation to the solution of the Cauchy problem (\ref{eq1}).
\subsection{Tikhonov regularization approach of the problem}
Using the proposed formulation $(\ref{pb7-1})$  of the Cauchy inverse problem (\ref{eq1}), we can suggest a Tikhonov regularization approach for the Cauchy problem with noisy Cauchy data. So let us consider the Cauchy data $(\psi,\kappa)$ containing the random noise with a known level $\delta$, denoted as $(\psi^\delta,\kappa^\delta),$ such that
$$
|||\psi^\delta-\psi|||_{0,\Gamma_0}\leq\delta, \qquad
|||\kappa^\delta-\kappa|||_{0,\Gamma_0}\leq \delta
$$
Then the boundary value problem (\ref{pb5}) associated to noisy Cauchy data can be read
\begin{equation}
\label{15}\left\lbrace
\begin{array}{llll}
-2\mu\,div(D(u^\delta))+\nabla p^\delta=f& in  &\Omega\\
div(u^\delta)=0 &on &\Omega\\
\sigma(u^\delta)n+iu^\delta=\psi^\delta+i\kappa^\delta & on & \Gamma_0\\
 \sigma(u^\delta)n+iu^\delta=\varphi+i\zeta & on &\Gamma_1\\
\end{array}\right.
\end{equation}
and its  weak formulation can be written
\begin{equation}
\label{weak15}
\left\lbrace
\begin{array}{ll}
\mbox{find} (u^\delta,p^\delta)\in V\times \Theta_0&\\
a(u^\delta,v)+b(v,p^\delta)=F^\delta(\varphi,\zeta,v)& \forall v\in V \\
b(u^\delta,q)=0&\forall q\in \Theta_0
\end{array}
\right.
\end{equation}
where
\begin{equation}
\label{16}
F^\delta(\varphi,\zeta,v)=\int_{\Omega}f\bar{v}dx+\int_{\Gamma_0}(\psi^\delta+i\kappa^\delta)\bar{v}ds+\int_{\Gamma_1}(\varphi+i\zeta)\bar{v}ds\qquad\forall v\in V
\end{equation}
 It is clear that the well-posedness  of the problem  (\ref{weak15}) follows from proposition~\ref{pro3.1},  indeed, we have the following result
\begin{prop}\label{pro3.1}
For given $f\in \Theta,$ $(\psi^\delta,\kappa^\delta)\in \Theta_{\Gamma_0}\times \Theta_{\Gamma_0}$ and $(\varphi,\zeta)\in \Theta_{\Gamma_1}\times \Theta_{\Gamma_1}$, the problem (\ref{weak15}) has a unique solution $(u^\delta, p^\delta)\in V\times \Theta_0$ which depends continuously on  data. Moreover
\begin{equation}
\label{18}
|||u^\delta|||_{1,\Omega}\leq \frac{c}{\alpha}(|||f|||_{0,\Omega}+|||\psi^\delta|||_{0,\Gamma_0}+|||\kappa^\delta|||_{0,\Gamma_0}+|||\varphi|||_{0,\Gamma_1}+|||\zeta|||_{0,\Gamma_0})
\end{equation}
\begin{equation}
\label{18.5}
||p^\delta||_{0,\Omega}\leq \frac{c}{\beta}(|||f|||_{0,\Omega}+|||\psi^\delta|||_{0,\Gamma_0}+|||\kappa^\delta|||_{0,\Gamma_0}+|||\varphi|||_{0,\Gamma_1}+|||\zeta|||_{0,\Gamma_0})
\end{equation}
\end{prop}
The error estimations between  the solutions of (\ref{weak10}) and (\ref{weak15}) with respect to $\delta$ is given by
\begin{lem}
Let $u$ and $u^\delta$ be the respective solution of (\ref{weak10}) and (\ref{weak15}), then we have
\begin{equation}
 \label{19}
 |||u^\delta-u|||_{1,\Omega}\leq c\delta
 \end{equation}
and
\begin{equation}
\label{press}
\|p^\delta-p\|_{0,\Omega}\leq c\delta
\end{equation}
\end{lem}
\begin{proof}
By subtracting the weak formulations (\ref{weak10}) and (\ref{weak15}), and putting $v=u^\delta-u$,   from the coercivity of $a$, the inf-sup condition and the continuity of $F,$ we get
\begin{equation}
\alpha|||u^\delta-u|||_{1,\Omega}^2+\beta |||u^\delta-u|||_{1,\Omega}\|p^\delta-p\|_{0,\Omega}\leq (|||\psi^\delta-\psi|||_{0,\Gamma_0}+|||\kappa^\delta-\kappa|||_{0,\Gamma_0})|||u^\delta-u|||_{1,\Omega},
\end{equation}
which end the proof.
\end{proof}
Now, for any $(\varphi,\zeta) \in \Theta_{\Gamma_1}\times \Theta_{\Gamma_1}$, we can write the solution $u^\delta(\varphi,\zeta)\in V$ of (\ref{weak15}), as follows
$$u^\delta(\varphi,\zeta)= u_1^\delta(\varphi,\zeta) + i\,u_2^\delta(\varphi,\zeta)$$
Define then the coast functional
\begin{equation}
J_\varepsilon^\delta(\varphi,\zeta) = \frac{1}{2}|||u_2^\delta(\varphi, \zeta)|||^2_{0,\Omega}+ \frac{1}{2}\|p_2^\delta(\varphi, \zeta)\|^2_{0,\Omega}+ \frac{\varepsilon}{2}|||\varphi|||^2_{0,\Gamma_1}+\frac{\varepsilon}{2}|||\zeta|||^2_{0,\Gamma_1}
\label{20}
\end{equation}
and introduce the following Tikhonov regularization approach for solving the problem (\ref{pb7-1})
\begin{equation}\left\lbrace
\begin{array}{l}
\mbox{find } \,\,(\varphi_\varepsilon^\delta,\zeta_\varepsilon^\delta) \in \Theta_{\Gamma_1}\times \Theta_{\Gamma_1}\,\, \mbox{ such that}\\
J_\varepsilon^\delta(\varphi_\varepsilon^\delta,\zeta_\varepsilon^\delta)  = \displaystyle\inf_{(\eta,s)\in  \Theta_{\Gamma_1}\times \Theta_{\Gamma_1}}
J_\varepsilon^\delta(\eta,s)
\end{array}\right.
\label{pb3.2}
\end{equation}
\subsection{Existence and convergence results}
In this section, we show first the existence of the optimal solution of $(\ref{pb3.2})$ and we identify it with respect to the solution of the adjoint state problem. Then we prove the convergence of subsequence of optimal solution of Tikhonov regularization formulation $(\ref{pb3.2})$ to the solution of the Cauchy problem (\ref{eq1}). For this, we will need the following lemmas. In the first lemma, we show the existence of the optimal solution of $(\ref{pb3.2})$.
\begin{lem}
\label{lemexist}
The problem (\ref{pb3.2}) admits a unique solution
\end{lem}
\begin{proof}
Let $\left(\varphi^n,\zeta^n\right)_n \in \Theta_{\Gamma_1}\times\Theta_{\Gamma_1}$ be a  minimizing sequence of $J^\delta_\varepsilon$ in $\Theta_{\Gamma_1}\times\Theta_{\Gamma_1}$, i.e.
\begin{equation}
\displaystyle \lim_{n\rightarrow +\infty} J^\delta_\varepsilon(\varphi^n,\zeta^n)=\inf_{(\eta,s)\in\Theta_{\Gamma_1}\times\Theta_{\Gamma_1}}J^\delta_\varepsilon(\eta,s)
\end{equation}
Since $J^\delta_\varepsilon$ is coercive, which means that $$\displaystyle\lim\limits_{|||(\varphi,\zeta)|||_{\Theta_{\Gamma_1}\times\Theta_{\Gamma_1}}\rightarrow+\infty}J^\delta_\varepsilon(\varphi,\zeta)=+\infty.$$
Thus $\left((\varphi^n,\zeta^n)\right)_n$ is bounded in $\Theta_{\Gamma_1}\times\Theta_{\Gamma_1}$. We can  then extract a subsequence denoted again $\left((\varphi^n,\zeta^n)\right)_n$ which converges  weakly in $\Theta_{\Gamma_1}\times\Theta_{\Gamma_1}$ to an element $(\tilde{\varphi},\tilde{\zeta})$
\begin{equation}
\label{seq}
(\varphi^n,\zeta^n)\underset{n\rightarrow +\infty}{\longrightarrow}(\tilde{\varphi},\tilde{\zeta})\,\,\mbox{in}\,\,\Theta_{\Gamma_1}\times\Theta_{\Gamma_1}
\end{equation}
According to proposition \ref{pro3.1}, $(u^n,p^n)= \left(u(\varphi^n,\zeta^n),p(\varphi^n,\zeta^n)\right)$ the solution of the weak formulation (\ref{weak10}) associated to
$(\varphi^n,\zeta^n)$ exist and satisfy
\begin{equation}
\label{cont}
|||u^n|||_{1,\Omega}\leq \frac{c}{\alpha}(|||f|||_{0,\Omega}+|||\psi|||_{0,\Gamma_0}+|||\kappa|||_{0,\Gamma_0}+|||\varphi^n|||_{0,\Gamma_1}+|||\zeta^n|||_{0,\Gamma_0})
\end{equation}
and
\begin{equation}
\label{contp}
\|p^n\|_{0,\Omega}\leq\frac{c}{\beta}(|||f|||_{0,\Omega}+|||\psi|||_{0,\Gamma_0}+|||\kappa|||_{0,\Gamma_0}+|||\varphi^n|||_{0,\Gamma_1}+|||\zeta^n|||_{0,\Gamma_0})
\end{equation}
which implies that
the sequence $(u^n)_n$ (respectively $(p^n)_n$  is bounded   in $V$ (respectively in $\Theta$). We can then extract a subsequence denoted again respectively $(u^n)_n$ and $(p^n)_n$ which converges weakly respectively in $V$ and $\Theta$.
So, we have
\begin{equation}
\label{uweak}
u^n\underset{n\rightarrow +\infty}{\longrightarrow}\tilde{u}\,\,\mbox{in}\,\,V
\end{equation}
and
\begin{equation}
\label{pweak}
p^n\underset{n\rightarrow +\infty}{\longrightarrow}\tilde{p}\,\,\mbox{in}\,\,\Theta
\end{equation}
It is then easy to verify  that $(\tilde{u},\tilde{p})=\left(u(\tilde{\varphi},\tilde{\zeta}),p(\tilde{\varphi},\tilde{\zeta})\right)$ the solution of (\ref{weak10}) associated to $(\tilde{\varphi},\tilde{\zeta})$. Indeed, since $(u^n,p^n)\in V\times \Theta_0$ is solution of
\begin{equation}
\label{lemm1}
\left\lbrace
\begin{array}{ll}
a(u^n,v)+b(v,p^n)=F((\varphi^n,\zeta^n),v) &\forall v\in V\\
b(u^n,q)=0&\forall q\in \Theta_0
\end{array}
\right.
\end{equation}
by passing to the limit in (\ref{lemm1}) as $n \rightarrow\infty$  and using the convergence (\ref{uweak}) ,(\ref{pweak}) and (\ref{seq}), we get
\begin{equation}
\label{lemm2}
\left\lbrace
\begin{array}{ll}
a(\tilde{u},v)+b(v,p^n)=F((\tilde{\varphi},\tilde{\zeta}),v) &\forall v\in V\\
b(\tilde{u},q)=0&\forall q\in \Theta_0
\end{array}
\right.
\end{equation}
i.e.  $(\tilde{u},\tilde{p})$ is the unique solution of (\ref{weak10}) associated $(\tilde{\varphi},\tilde{\zeta})$.  Thus we have the following weak convergence
\begin{equation}
u_2^n=u_2(\varphi^n,\zeta^n)\underset{n\rightarrow +\infty}{\longrightarrow} \tilde{u}_2(\varphi,\zeta)\,\,\mbox{in}\,\,\Theta
\label{u2weak}
\end{equation}
\begin{equation}
p_2^n=p_2(\varphi^n,\zeta^n)\underset{n\rightarrow +\infty}{\longrightarrow} \tilde{p}_2(\varphi,\zeta)\,\,\mbox{in}\,\,\Theta
\label{p2weak}
\end{equation}
Then according to the lower semi-continuity of the norms in $L^2$, we obtain
\begin{equation}
\inf_{(\eta,s)\in\Theta_{\Gamma_1}\times\Theta_{\Gamma_1}}J^\delta_\varepsilon(\eta,s)=\displaystyle \liminf\limits_{n\rightarrow +\infty}J^\delta_\varepsilon(\varphi^n,\zeta^n)\geq J^\delta_\varepsilon(\tilde{\varphi},\tilde{\zeta})
\end{equation}
Therefore, the problem (\ref{pb3.2}) admits at least one solution solution. Then it is easy to see that $J^\delta_\varepsilon$ is strictly convex, which permits to conclude that the solution of the problem (\ref{pb3.2}) is unique.
\end{proof}
The gradient formula of $J^\delta_\varepsilon$ is given in the following result.
\begin{lem}
Let $(\varphi,\zeta)\in \Theta_{\Gamma_1}\times \Theta_{\Gamma_1}$, then for all $(\eta,s)\in \Theta_{\Gamma_1}\times \Theta_{\Gamma_1}$ the gradient of $ J_\varepsilon^\delta$ is given by
$$\nabla J_\varepsilon^\delta(\varphi,\zeta)(\eta,s)=(u_2^\delta(\varphi, \zeta),u_2^\delta(\eta,s)-u^\delta_2(0,0))_{0,\Omega}+(p_2^\delta(\varphi, \zeta),p_2^\delta(\eta,s)-p^\delta_2(0,0))_{0,\Omega}+\varepsilon(\varphi,\eta)_{0,\Gamma_1}+\varepsilon(\zeta,s)_{0,\Gamma_1}$$
\end{lem}
\begin{proof}
Let $(\varphi,\zeta)$ and $(\eta,s)$ in $\Theta_{\Gamma_1}\times \Theta_{\Gamma_1}$, and let $\tau $ be a real which it will tend to $0$. We consider
$\left(u_2^\delta(\varphi+\tau\eta,\zeta+\tau s),p_2^\delta(\varphi+\tau\eta,\zeta+\tau s)\right)$ the solution of
\begin{equation}\left\lbrace
\begin{array}{llll}
-2\mu\,div(D(u_2^\delta))+\nabla p_2^\delta=0&in  &\Omega\\
div(u_2^\delta)=0 &on &\Omega\\
\sigma(u_2^\delta)n+u_1^\delta=\kappa^\delta & on & \Gamma_0\\
 \sigma(u_2^\delta)n+u_1^\delta=\zeta+\tau s & on &\Gamma_1\\
\end{array}\right.
\end{equation}
where $\left(u_1^\delta(\varphi+\tau\eta,\zeta+\tau s), p_1^\delta(\varphi+\tau\eta,\zeta+\tau s)\right)$ is the solution of
\begin{equation}\left\lbrace
\begin{array}{llll}
-2\mu\,div(D(u_1^\delta))+\nabla p_1^\delta=f&in  &\Omega\\
div(u_1^\delta)=0 &on &\Omega\\
\sigma(u_1^\delta)n-u_2^\delta=\psi^\delta & on & \Gamma_0\\
 \sigma(u_1^\delta)n-u_2^\delta=\varphi+\tau\eta & on &\Gamma_1\\
\end{array}\right.
\label{pb6}
\end{equation}
It is then easy to see that
\begin{equation}
u_2^\delta(\varphi+\tau\eta,\zeta+\tau s)=u_2^\delta(\varphi,\zeta)+\tau\left(u_2^\delta(\eta,s)-u_2^\delta(0,0)\right)
\label{u2lin}
\end{equation}
and
\begin{equation}
p_2^\delta(\varphi+\tau\eta,\zeta+\tau s)=p_2^\delta(\varphi,\zeta)+\tau\left(p_2^\delta(\eta,s)-p_2^\delta(0,0)\right)
\label{p2lin}
\end{equation}
By using the equation (\ref{u2lin}) and (\ref{p2lin}), we get
\begin{equation}
\left( u_2^\delta(\varphi+\tau\eta,\zeta+\tau s),u_2^\delta(\varphi+\tau\eta,\zeta+\tau s)\right)_{0,\Omega}-\left( u_2^\delta(\varphi,\zeta),u_2^\delta(\varphi,\zeta)\right)_{0,\Omega}=(u_2^\delta(\varphi, \zeta),u_2^\delta(\eta,s)-u^\delta_2(0,0))_{0,\Omega}
\label{q1}
\end{equation}
and
\begin{equation}
\left( p_2^\delta(\varphi+\tau\eta,\zeta+\tau s),p_2^\delta(\varphi+\tau\eta,\zeta+\tau s)\right)_{0,\Omega}-\left(p_2^\delta(\varphi,\zeta),p_2^\delta(\varphi,\zeta)\right)_{0,\Omega}=(p_2^\delta(\varphi, \zeta),p_2^\delta(\eta,s)-p^\delta_2(0,0))_{0,\Omega}
\label{q2}
\end{equation}
Thus the formula of the gradient of $ J_\varepsilon^\delta$ follows from the equations $(\ref{q1})$ and $(\ref{q2})$.
\end{proof}
We can then state the following existence result.
\begin{prop}
For any $\varepsilon>0$, the unique  solution $(\varphi_\varepsilon^\delta,\zeta_\varepsilon^\delta) \in \Theta_{\Gamma_1}\times \Theta_{\Gamma_1}$ of the problem (\ref{pb3.2})  is given  by
\begin{equation}
\label{phi_t}
\varphi_\varepsilon^{\delta}=-\frac{1}{\varepsilon}w{^\delta}_{2|\Gamma_1}\qquad
\zeta_\varepsilon^{\delta}=-\frac{1}{\varepsilon}w^{\delta}_{1|\Gamma_1},
\end{equation}
where $w_1^\delta=w_1^\delta(\varphi_\varepsilon^\delta,\zeta_\varepsilon^\delta)$ and $w_2^\delta=w_2^\delta(\varphi_\varepsilon^\delta,\zeta_\varepsilon^\delta)$ are the real and imaginary parts of the  solution $w^\delta$ of the adjoint state boundary value problem:
\begin{equation}
\label{15}\left\lbrace
\begin{array}{llll}
-2\mu\,div(D(w^\delta))+\nabla \tilde{p}^\delta=u_2^\delta & in  &\Omega\\
div(w^\delta)=-p_2^\delta &on &\Omega\\
\sigma(w^\delta)n+iw^\delta=0 & on & \Gamma\\
\end{array}\right.
\end{equation}
and $u_2^\delta=u_2^\delta(\varphi_\varepsilon^{\delta},\zeta_\varepsilon^\delta),\quad p_2^\delta=p_2^\delta(\varphi_\varepsilon^{\delta},\zeta_\varepsilon^\delta)$ are the imaginary parts of the solution $(u^\delta,p^\delta)$ of the problem \ref{weak15}, with $(\varphi,\zeta)$ being replaced by $(\varphi_\varepsilon^\delta,\zeta_\varepsilon^\delta).$
\label{pp3.3}
\end{prop}
\begin{proof}
For any $\varepsilon>0,$ 
the solution $(\varphi_{\varepsilon}^{\delta},\zeta_{\varepsilon}^{\delta})$ is characterized by
\begin{equation}
\label{23}
\nabla J_\varepsilon^\delta(\varphi_{\varepsilon}^{\delta},\zeta_{\varepsilon}^{\delta})(\eta,s)=0
\end{equation}
In other side, let $\tilde{u}=u^\delta(\eta,s)-u^\delta(0,0)$ and $q=p^\delta(\eta,s)-p^\delta(0,0),$
so that  $(\tilde{u},q)$ is the solution of
\begin{equation}\left\lbrace
\begin{array}{llll}
-2\mu\,div(D(\tilde{u}))+\nabla q=0&on  &\Omega\\
div(\tilde{u})=0 &on &\Omega\\
\sigma(\tilde{u})n=0,\qquad \tilde{u}=0& on & \Gamma_0\\
 \sigma(\tilde{u})n=\eta,\qquad \tilde{u}=s & on &\Gamma_1\\

\end{array}\right.
\end{equation}
which can be written with Robin boundary conditions as
\begin{equation}
\label{15}\left\lbrace
\begin{array}{llll}
-2\mu\,div(D(\tilde{u}))+\nabla q=0& in  &\Omega\\
div(\tilde{u})=0 &on &\Omega\\
\sigma(\tilde{u})n+i\tilde{u}=0 & on & \Gamma_0\\
 \sigma(\tilde{u})n+i\tilde{u}=\eta+is & on &\Gamma_1\\

\end{array}\right.
\end{equation}
By multiplying the adjoint equation by $(\tilde{u},q)$ and integrating  on $ \Omega$, we have
\begin{align*}
-2\mu\int_{\Omega}div(D(w^\delta))\tilde{u}dx+\int_{\Omega}\nabla \tilde{p}^\delta\tilde{u}dx&=\int_{\Omega}u_2^\delta \tilde{u}dx\\
\int_{\Omega} div w^\delta qdx=-\int_{\Omega} p_2^\delta q dx.
\end{align*}
Using then the Green formulas, we obtain
$$
2\mu\int_{\Omega}D(w^\delta):D(\tilde{u})dx-2\mu\int_{\Gamma}D(w^\delta)n\tilde{u} ds+\int_{\Gamma}\tilde{p}n\tilde{u} ds-\int_{\Omega}\tilde{p}^\delta div(\tilde{u})dx=\int_{\Omega}u_2^\delta\tilde{u} dx
$$
$$-\int_{\Omega}\nabla q w^\delta\,dx+\int_{\Gamma}w^\delta\,qn\,ds=-\int_{\Omega}p_2^\delta qdx$$
Then

\begin{equation}
-2\mu\int_{\Omega}div(D(\tilde{u}))w^\delta dx+2\mu\int_{\Omega} D(\tilde{u}) n w^\delta dx-2\mu\int_{\Gamma}D(w^\delta)n\tilde{u} dx+\int_{\Gamma}\tilde{p}^\delta n\tilde{u}ds-\int_{\Omega}\tilde{p}^\delta div(\tilde{u})dx=\int_{\Omega}u_2^\delta\tilde{u} dx
\label{eq*}
\end{equation}
and
\begin{equation}
-\int_{\Omega}\nabla q w^\delta\,dx+\int_{\Gamma}w^\delta\,q\,n\,ds=-\int_{\Omega}p_2^\delta q dx
\label{eq**}
\end{equation}
by subtracting  $(\ref{eq*})$ and $(\ref{eq**}$), we obtain
\begin{align*}
-2\mu\int_{\Omega}div(D(\tilde{u}))w^\delta dx+\int_{\Omega}\nabla q w^\delta dx+\int_{\Gamma} \sigma(\tilde{u})n w^\delta ds-\int_{\Gamma}\sigma(w^\delta)n\tilde{u}ds-\int_{\Omega}\tilde{p}^\delta div(\tilde{u})dx =\int_{\Omega}u_2^\delta \tilde{u}dx+\int_{\Omega}p_2^\delta q\,dx
\end{align*}
Thus
$$\int_{\Gamma}\sigma(\tilde{u})nw^\delta ds-\int_{\Gamma}\sigma(w^\delta)n\tilde{u}ds=\int_{\Omega}u_2^\delta\tilde{u}dx+\int_{\Omega}p_2^\delta q\,dx$$
therefore
$$\int_{\Gamma_1}(\eta+is)w^\delta\,ds=\int_{\Omega}u_2^\delta\tilde{u}dx +\int_{\Omega}p_2^\delta q\,dx$$
By replacing the last expression in the gradient formula, we obtain
$$(\eta+is,w^\delta)_{0,\Gamma_1}+\varepsilon(\varphi,\eta)_{0,\Gamma_1}+\varepsilon(\zeta,s)_{0,\Gamma_1}=0$$
Then,
$$(\eta,w_2^\delta+\varepsilon\varphi)_{0,\Gamma_1}+(s,w_1^\delta+\varepsilon \zeta)_{0,\Gamma_1}=0.$$
By taking $\eta=w_2^\delta+\varepsilon\varphi$ and $s=w_1^\delta+\varepsilon \zeta$,
we get  the formula (\ref{phi_t}).
\end{proof}

Now, let us study the behavior of $(\varphi_\varepsilon^\delta,\zeta_\varepsilon^\delta)$ as $\delta\rightarrow 0$ and $\varepsilon\rightarrow 0.$ To do that, we assume that the Cauchy data $(\psi, \kappa)$ are compatible. Then according to \cite{izakov, lions}, the problem $(\ref{eq1})$ admits a unique solution $(\varphi^*, \zeta^*)\in  H^{-\frac{1}{2}}(\Gamma_1)\times H^\frac{1}{2}(\Gamma_1)$.

For a sequence of noise levels $\{\delta_n\}_{n\geq 1}$ which converges to $0$ in $\mathbb{R}$ as $n\rightarrow\infty$, let $\varepsilon_n = \varepsilon(\delta_n)$ be chosen satisfying $\varepsilon_n \rightarrow 0$ and $\frac{\delta^2_n}{\varepsilon_n}\rightarrow 0$, as $n\rightarrow\infty$. Denote by $(\varphi^{\delta_n}_{\varepsilon_n}, \zeta^{\delta_n}_{\varepsilon_n}) \in \Theta_{\Gamma_1}\times \Theta_{\Gamma_1}$
the solution of $(\ref{pb3.2})$ associated to $(\psi^\delta, \kappa^\delta)$
 and $\varepsilon$ replaced by $(\psi^{\delta_n}, \kappa^{\delta_n})$ and $\varepsilon_n$ respectively, and moreover assume  that $\varphi^*$ belongs to $\Theta_{\Gamma_1}$. We have then the following result:
\begin{prop}
There exists a subsequence of solution of $(\ref{pb3.2})$ denoted again $\{(\varphi^{\delta_n}_{\varepsilon_n}, \zeta^{\delta_n}_{\varepsilon_n})\}_{n}$   which converges to $(\varphi^*,\zeta^*)$ in $\Theta_{\Gamma_1} \times \Theta_{\Gamma_1}$ as $n\rightarrow\infty$.
\end{prop}
\begin{proof}
For simplicity, let us denote by $\psi^n =\psi^{\delta_n} ,\kappa^n =\kappa^{\delta_n}, \varphi^n =\varphi^{\delta_n}_{\varepsilon_n}$, $\zeta^n =\zeta^{\delta_n}_{\varepsilon_n}$ and  by $u^n=u_1^n+i\,u^n_2~=~u^\delta_n (\varphi_n, \zeta_n)\in V$  the solution of (\ref{weak15}) associated to $(\varphi_n, \zeta_n)$. It is clear  that $(\varphi^*, \zeta^*)$,  the unique solution of $(\ref{eq1})$,  is also the unique solution of $(\ref{pb5})$, due to the equivalence of the two problems, and thus $$u_2(\varphi^*, \zeta^*) = 0 \,\,\mbox{in} \,\,\Omega,\qquad p_2(\varphi^*, \zeta^*) = 0 \,\,\mbox{in} \,\,\Omega,$$
where $(u_2(\varphi^*,\zeta^*),p_2(\varphi^*,\zeta^*))$   is the imaginary part of the solution of the problem (\ref{weak10}) associated to $(\varphi, \zeta)$ replaced by $(\varphi^*, \zeta^*)$. \\
Therefore, using the fact that  $(\varphi_n, \zeta_n)$ is solution of  $(\ref{pb3.2})$ and the inequalities (\ref{19}) and (\ref{press}), we get
\begin{equation}
\label{25}
\begin{array}{ll}
\displaystyle J^{\delta_n}_{\varepsilon_n}(\varphi^n,\zeta^n)\leq J^{\delta_n}_{\varepsilon_n}(\varphi^*,\zeta^*)& =\displaystyle \frac{1}{2}|||u^{\delta_n}_2(\varphi^*,\zeta^*)|||_{0,\Omega}^2+\frac{1}{2}\|p^{\delta_n}_2(\varphi^*,\zeta^*)\|_{0,\Omega}^2+\frac{\varepsilon_n}{2}|||\varphi^*|||_{0,\Gamma_1}^2+\frac{\varepsilon_n}{2}|||\zeta^*|||_{0,\Gamma_1}^2 \\
& =\displaystyle\frac{1}{2}|||u^{\delta_n}_2(\varphi^*,\zeta^*)-u_2(\varphi^*,\zeta^*)|||_{0,\Omega}^2+\frac{1}{2}\|p^{\delta_n}_2(\varphi^*,\zeta^*)-p_2(\varphi^*,\zeta^*)|||_{0,\Omega}^2\\
&\qquad\displaystyle +\frac{\varepsilon_n}{2}|||\varphi^*|||_{0,\Gamma_1}^2+\frac{\varepsilon_n}{2}|||\zeta^*|||_{0,\Gamma_1}^2 \\
&\leq  \displaystyle c\delta_n^2+\frac{\varepsilon_n}{2}|||\varphi^*|||_{0,\Gamma_1}^2+\frac{\varepsilon_n}{2}|||\zeta^*|||_{0,\Gamma_1}^2.
\end{array}
\end{equation}
Then
\begin{equation}
\label{26}
|||\varphi^n|||^2_{0,\Gamma_1}+|||\zeta^n|||_{0,\Gamma_1}^2\leq c \frac{\delta_n^2}{\varepsilon_n}+|||\varphi^*|||_{0,\Gamma_1}^2+|||\zeta^*|||_{0,\Gamma_1}^2.
\end{equation}
Moreover, from (\ref{18}) and (\ref{18.5}), we get
\begin{equation}
\label{27}
\begin{array}{ll}
\displaystyle|||u^n |||_{1,\Omega}&\leq \frac{c}{\alpha}\left( |||f|||_{0,\Omega}+|||\psi^n|||_{0,\Gamma_0}+|||\kappa^n|||_{0,\Gamma_0}+|||\varphi^n|||_{0,\Gamma_1}+|||\zeta^n|||_{0,\Gamma_1}\right) \\
&\leq \frac{c}{\alpha}\left(|||f|||_{0,\Omega}+2\delta_n+|||\psi|||_{0,\Gamma_0}+|||\kappa|||_{0,\Gamma_0}+|||\varphi^n|||_{0,\Gamma_1}+|||\zeta^n|||_{0,\Gamma_1}\right)
\end{array}
\end{equation}
and
\begin{equation}
\label{27.5}
\|p^n\|_{0,\Omega}\leq \frac{c}{\beta}\left(|||f|||_{0,\Omega}+2\delta_n+|||\psi|||_{0,\Gamma_0}+|||\kappa|||_{0,\Gamma_0}+|||\varphi^n|||_{0,\Gamma_1}+|||\zeta^n|||_{0,\Gamma_1}\right).
\end{equation}
Therefore, by combining (\ref{26}), (\ref{27}) and (\ref{27.5}), for $n$ large enough, the sequence $\{(\varphi^n,\zeta^n,u^n,p^n)\}_n $ is uniformly bounded  in $\Theta_{\Gamma_1}\times \Theta_{\Gamma_1}\times V\times\Theta,$ with respect to $n$. We can then extract a subsequence denoted again $\{(\varphi^n,\zeta^n,u^n,p^n)\}_n $ which converges weakly to  some elements $\left(\tilde{\varphi},\tilde{\zeta}\right)\in \Theta_{\Gamma_1}\times \Theta_{\Gamma_1}, \tilde{u}\in V \mbox{ and } \tilde{p} \in \Theta_0$ such that as $n\rightarrow \infty$
\begin{equation}
\label{conv}
\begin{array}{l}
\{(\varphi^n,\zeta^n)\}_n\rightarrow \left(\tilde{\varphi},\tilde{\zeta}\right) \mbox{ in }\quad \Theta_{\Gamma_1}\times \Theta_{\Gamma_1} \\
u^{n}\rightarrow\tilde{u} \mbox{ in } V,\,\, u^n\rightarrow \tilde{u} \mbox{ in } \Theta,\,\, u^n\rightarrow \tilde{u} \mbox{ in } \Theta_{\Gamma}\\
p^{n}\rightarrow \tilde{p} \mbox{ in } \Theta.
\end{array}
\end{equation}
Using the same technics as in lemma~\ref{lemexist}, it is then easy to verify  that $(\tilde{u},\tilde{p})=\left(u(\tilde{\varphi},\tilde{\zeta}),p(\tilde{\varphi},\tilde{\zeta})\right)$ the solution of (\ref{weak10}) associated $(\tilde{\varphi},\tilde{\zeta})$. 

On the other hand, according the the lower semi-continuity of the norms in $L^2$, and using the fact that $\varepsilon_n\rightarrow 0$ as $n\rightarrow\infty$, we obtain
\begin{equation}
\label{31}
\displaystyle \liminf_{n \rightarrow\infty} J^{\delta_n}_{\varepsilon_n}(\varphi^n,\zeta^n)=\liminf_{n \rightarrow\infty}\frac{1}{2}|||u^n_2|||_{0,\Omega}^2+\frac{1}{2}\|p^n_2\|_{0,\Omega}^2+\frac{\varepsilon_n}{2}|||\varphi^*|||_{0,\Gamma_1}^2+\frac{\varepsilon_n}{2}|||\zeta^*|||_{0,\Gamma_1}^2\geq \frac{1}{2}|||\tilde{u}_2|||_{0,\Omega}^2+\frac{1}{2}\|\tilde{p}_2\|_{0,\Omega}^2
\end{equation}
From \ref{25}, we have
\begin{equation}
0\leq J^{\delta_n}_{\varepsilon_n}(\varphi^n,\zeta^n)\leq J^{\delta_n}_{\varepsilon_n}(\varphi^*,\zeta^*)\leq c\delta_n^2+\frac{\varepsilon_n}{2}|||\varphi^*|||_{0,\Gamma_1}^2+\frac{\varepsilon_n}{2}|||\zeta^*|||_{0,\Gamma_1}^2
\label{32}
\end{equation}
This implies that
\begin{equation}
\displaystyle \lim_{n\rightarrow\infty}J^{\delta_n}_{\varepsilon_n}(\varphi^n,\zeta^n)=0.
\label{32.5}
\end{equation}
Then from  (\ref{31}) and (\ref{32.5}) we  get
$$\tilde{u}_2=0\qquad \tilde{p}_2=0   \mbox{ in } \Omega,$$
which shows that $(\tilde{\varphi},\tilde{\zeta})\in \Theta_{\Gamma_1}\times \Theta_{\Gamma_1}$ is a solution of the problem (\ref{pb3.2}). Since $(\varphi^*,\zeta^*)$ is the unique solution of the problem  (\ref{pb3.2}), we conclude that $(\tilde{\varphi},\tilde{\zeta})=(\varphi^*,\zeta^*).$
Moreover the sequence $\{(\varphi^n,\zeta^n)\}_n$ converges to $(\varphi^*,\zeta^*)$ in $\Theta_{\Gamma_1}\times \Theta_{\Gamma_1}$   as $n\rightarrow \infty.$ Indeed, by using (\ref{26}) and the weak convergence (\ref{conv}), we have
\begin{eqnarray*}
\displaystyle
|||\varphi^n-\varphi^*|||^2_{0,\Gamma_1}+|||\zeta^n-\zeta^*|||^2_{0,\Gamma_1}
&=&\displaystyle|||\varphi^n|||^2_{0,\Gamma_1}+|||\zeta^n|||_{0,\Gamma_1}^2+|||\varphi^*|||^2_{0,\Gamma_1}+|||\zeta^*|||^2_{0,\Gamma_1}-2(\varphi^n,\varphi^*)_{0,\Gamma_1}-2(\zeta^n,\zeta^*)_{0,\Gamma_1}\\
&\leq& \displaystyle c\frac{\delta_n^2}{\varepsilon_n}+2|||\varphi^*|||^2_{0,\Gamma_1}+2|||\zeta^*|||^2_{0,\Gamma_1}-2(\varphi^n,\varphi^*)_{0,\Gamma_1}-2(\zeta^n,\zeta^*)_{0,\Gamma_1}
\end{eqnarray*}
Then we get the strong convergence
\begin{equation}
\displaystyle \lim_{n\rightarrow}\left(|||\varphi^n-\varphi^*|||^2_{0,\Gamma_1}+|||\zeta^n-\zeta^*|||^2_{0,\Gamma_1}\right)=0.
\label{32.6}
\end{equation}
Which achieves the proof.
\end{proof}

\section{Approximation of the regularized optimal problem}
In this section,  we propose the approximation of the optimal solution of the regularization problem via finite elements method of type  $P_{1Bubble}/P_1$ \cite{finite}. For this,  we try first  to generate a non-iterative algorithm for solving the regularized optimal problem. In fact, the solution of the optimization problem \ref{pb3.2} is obtained by solving some linear systems. Indeed, from the proposition \ref{pp3.3}, we can construct the following algorithm:

 We substitute (\ref{phi_t}) into (\ref{weak15}), to get

\begin{equation}
\label{33}
\left\lbrace
\begin{array}{ll}
a(u^\delta,v)+b(v,p^\delta)+\frac{1}{\varepsilon}(w_2^\delta,\bar{v})_{0,\Gamma_1}+i\frac{1}{\varepsilon}(w_1^\delta,\bar{v})_{0,\Gamma_1}=(f,\bar{ v})+(\psi^\delta+i\kappa^\delta,\bar{v})_{0,\Gamma_0}& \forall v\in V\\
b(u^\delta,q)=0&\forall q\in \Theta_0
\end{array}
\right.
\end{equation}
Also the weak formulation  of (\ref{15}) can be written:
\begin{equation}
\label{34}
\left\lbrace
\begin{array}{ll}
a(w^\delta,v)+b(v,\tilde{p}^\delta)=(u^\delta_2,\bar{ v})_{0,\Omega}& \forall v\in V\\
b(w^\delta,q)=0&\forall q\in \Theta_0
\end{array}
\right.
\end{equation}
Then by combining (\ref{phi_t}), (\ref{33}) and (\ref{34}), we can summarized the following solver of the optimization problem (\ref{pb3.2}):

We solve first the following problem

\begin{equation}
\label{35}
\left\lbrace
\begin{array}{l}

2\mu\displaystyle\int_{\Omega}D(u_1):D(v)dx+\displaystyle\int_{\Omega}p_1div\,v\,dx-(u_2,v)_{0,\Gamma}+\frac{1}{\varepsilon}(w_2,v)_{0,\Gamma_1}=(f,v)_{0,\Omega}+(\psi^\delta,v)_{0,\Gamma_0}\qquad \forall v\in V\\
\displaystyle\int_{\Omega}q\,div u_1\,dx=0\qquad\forall q\in \Theta_0\\

2\mu\displaystyle\int_{\Omega}D(u_2):D(v)dx+\displaystyle\int_{\Omega}p_2div\,v\,dx+(u_1,v)_{0,\Gamma}+\frac{1}{\varepsilon}(w_1,v)_{0,\Gamma_1}=(\kappa^\delta,v)_{0,\Gamma_0}\quad \forall v\in V\\
\displaystyle\int_{\Omega}q\,div u_2\,dx=0\qquad\forall q\in \Theta_0\\

2\mu\displaystyle\int_{\Omega}D(w_1):D(v)dx+\int_{\Omega}\tilde{p_1} div\,v\,dx-(u_2,v)_{0,\Omega}-(w_2,v)_{0,\Gamma}=0\qquad \forall v\in V\\
\displaystyle\int_{\Omega}q\,div w_1\,dx=-p_2\qquad\forall q\in \Theta_0\\

2\mu\displaystyle\int_{\Omega}D(w_2):D(v)dx-\int_{\Omega}\tilde{p_2} div\,v\,dx+(w_1,v)_{0,\Gamma}=0 \qquad\forall v\in V\\
\displaystyle\int_{\Omega}q\,div w_2\,dx=0\qquad\forall q\in \Theta_0\\
\end{array}
\right.
\end{equation}

Then we compute 
\begin{equation}
\label{36}
\varphi^\delta_\varepsilon=-\frac{1}{\varepsilon}w_{2/\Gamma_1},\qquad   \zeta_\varepsilon^\delta=-\frac{1}{\varepsilon}w_{1/\Gamma_1}
\end{equation}

In the sequel, for solving numerically this optimization problem,  the standard conforming linear finite element methods for mixed formulations are applied  to solve \ref{35}, we use  for example $P_{1Bubble}/P_1.$ More precisely, let $\{\mathcal{T}_h\}_h$ be a regular  family of finite element partitions of $\bar{\Omega}.$  We designate by $\lambda_i^{(T)}$ ,$1\leq i\leq 3$ the barycentric coordinates with respect to the vertices of $T$. The  "Bubble" function  $b^{(T)}$ associated to the  triangle T is defined by
$$b^{(T)}=\prod\limits_{i=1}^{3}\lambda_i^{(T)}$$
The function $b$ is in fact a polynomial function of degree 3 which vanishes on the edges of $T.$ We define the space associated with the bubble function by
$$\mathbb{B}_h =\left\{v_h \in C(\bar{\Omega});\, \forall T \in  T_h \,\, v_{h|T} =xb^{(T)}\right\}.$$
We define next the functional spaces
$$V_i^{h} =\left\{v_h \in C(\bar{\Omega});\,\,v_{h|T}\in P^1\,\, \forall T \in  T_h\,\right\} \oplus B_h$$
$$\Theta^h =\left\{q_h \in C(\bar{\Omega});\,\,q_{h|T}\in P^1 \,\,\forall T \in  T_h\right\}$$
we note that
$$V^h=V_1^{h}\times V_2^{h}$$

Then  the finite element discretization of (\ref{35}) reads:

\begin{equation}
\label{37}
\left\lbrace
\begin{array}{l}

2\mu\displaystyle\int_{\Omega}D(u^h_1):D(v^h)dx+\displaystyle\int_{\Omega}p^h_1div\,v^h\,dx-(u^h_2,v^h)_{0,\Gamma}+\frac{1}{\varepsilon}(w^h_2,v^h)_{0,\Gamma_1}=(f,v^h)_{0,\Omega}+(\psi^\delta,v^h)_{0,\Gamma_0}\quad \forall v^h\in V^h\\
\displaystyle\int_{\Omega}q^h\,div u^h_1\,dx=0\qquad\forall q^h\in \Theta_0^h\\

2\mu\displaystyle\int_{\Omega}D(u^h_2):D(v^h)dx+\displaystyle\int_{\Omega}p^h_2div\,v^h+(u^h_1,v^h)_{0,\Gamma}+\frac{1}{\varepsilon}(w^h_1,v^h)_{0,\Gamma_1}=(\kappa^\delta,v^h)_{0,\Gamma_0}\quad \forall v^h\in V^h\\
\displaystyle\int_{\Omega}q^h\,div u^h_2\,dx=0\qquad\forall q^h\in \Theta_0^h\\

2\mu\displaystyle\int_{\Omega}D(w^h_1):D(v^h)dx+\int_{\Omega}\tilde{p_1}^h div\,v^h\,dx-(u^h_2,v^h)_{0,\Omega}-(w^h_2,v^h)_{0,\Gamma}=0\qquad \forall v^h\in V^h\\
\displaystyle\int_{\Omega}q^h\,div w^h_1\,dx=-\int_{\Omega}q^h\,p^h_2\qquad\forall q^h\in \Theta_0^h\\

2\mu\displaystyle\int_{\Omega}D(w^h_2):D(v^h)dx-\int_{\Omega}\tilde{p_2}^h div\,v^h+(w^h_1,v^h)_{0,\Gamma}=0 \qquad\forall v^h\in V^h\\
\displaystyle\int_{\Omega}q\,div w^h_2\,dx=0\qquad\forall q^h\in \Theta_0^h\\
\end{array}
\right.
\end{equation}
and the approximate value of the solution of (\ref{36}) is written 
\begin{equation}
\label{38}
\varphi^h_\varepsilon=-\frac{1}{\varepsilon}w^h_{2/\Gamma_1},\qquad   \zeta^h_\varepsilon=-\frac{1}{\varepsilon}w^h_{1/\Gamma_1}
\end{equation}

 Then we set $V^h=V^h+iV^h$, and define
$$\Theta_{\Gamma_1}^h=\{g^h\in \Theta_{\Gamma_1}\,\,|\,\, \exists v^h\in V^h \mbox{ such that}\,\, g^h=v^h_{|\Gamma_1} \}.$$
Thus it is easy to verify, using the same argument as in the continuous case, that
$$(\varphi_\varepsilon^h,\zeta_\varepsilon^h)=\underset{(\eta^h,s^h)\in \Theta^h_{\Gamma_1}\times \Theta^h_{\Gamma_1}}{argmin} J_{\varepsilon}^h(\eta^h,s^h)$$
with
$$J_\varepsilon^h(\varphi,\zeta)=\frac{1}{2}|||u_2^h(\varphi, \zeta)|||^2_{0,\Omega}+\frac{1}{2}\|p_2^h(\varphi, \zeta)\|^2_{0,\Omega}+ \frac{\varepsilon}{2}|||\varphi|||^2_{0,\Gamma_1}+\frac{\varepsilon}{2}|||\zeta|||^2_{0,\Gamma_1},$$
where $(u_2^h(\varphi,\zeta),p_2^h(\varphi,\zeta))$ is the imaginary part of the solution $(u^h,p^h)\in V^h\times\Theta_0^h$ of
\begin{equation}
\left\lbrace
\begin{array}{l}
a(u^h,v^h)+b(v^h,p^h)=F(\varphi,\zeta;v^h) \qquad\forall v^h\in V^h \\
b(u^h,q^h)=0\qquad\qquad\forall  q^h\in \Theta^h_0(\Omega)
\end{array}
\right.
\end{equation}
\section{Numerical results}
This section deals with the numerical realization of the proposed approach for solving a Stokes inverse Cauchy problem. 
In order to confirm the accurate of the coupled complex boundary problem based on Tikhononv regularization, were going to consider a Stokes inverse Cauchy problem defined by some data which allow us to gets the analytic solution.  For this, we consider the following data: 
\begin{itemize}
\item let $\Omega\subset\mathbb{R}^2$ be a ring defined by 
\[
\{(x,y)\in\R^2, / r^2<x^2+y^2<R^2\}
\]
where $r$ is the inner radius and $R$ is the external radius. 
\item The boundary of $\Omega$ is defined by $\partial\Omega:=\Gamma_0\cup\Gamma_1$. Where $\Gamma_1=\{(x,y)\in\R^2, / x^2+y^2=r^2\}$ is the internal boundary which is inaccessible and $\Gamma_0=\{(x,y)\in\R^2, / x^2+y^2=R^2\}$ is the external boundary. 
\item Dirichlet data $\kappa$ for the  velocity  field over $\Gamma_{0}$:

\begin{equation}
\left\lbrace
\begin{array}{l}
 \kappa_1=ch(x)\,sh(y)\\
 \kappa_2=-ch(y)\,sh(x)\\
 \end{array}
 \right.
\end{equation}
\item The flux $\psi$ over $\Gamma_0$
\begin{equation}
\left\lbrace
\begin{array}{l}
 \psi_1=\mu\,(x\,sh(x)\,sh(y)+y\,ch(x)\,ch(y))-x\,pe\\
 \psi_2=\mu\,(-x\,ch(x)\,ch(y)-y\,sh(y)\,sh(x))-y\,pe\\
 \end{array}
 \right.
\end{equation}
\item The second member
\begin{equation}
\left\lbrace
\begin{array}{l}
  f_1=-2\mu\,ch(x)sh(y)+ysh(x)\\
 f_2=2\mu\,ch(y)sh(x)+ch(x)\\
  \end{array}
 \right.
\end{equation}
\item  We try to reconstruct $(\varphi_\varepsilon^h,\zeta_\varepsilon^h)$ the approximation of velocity $\varphi^*$ and the pressure $\zeta^*$ and the flux  on $\Gamma_1,$ by considering the following exact solution of the problem (\ref{eq1}) associated to the above data
\begin{equation}
\left\lbrace
\begin{array}{l}
 ue1=\mbox{ch(x)sh(y)}\\
 ue2=-\mbox{ch(y)sh(x)}\\
 pe=\mbox{ych(x)-}\frac{\mbox{sh(1)}}{2}
\end{array}
\right.
\end{equation}
\end{itemize}
 First, we study the influence of the parameter of regularisation of Tikhonov $\varepsilon$ on the approximate solution $(\varphi_\varepsilon^h,\zeta_\varepsilon^h)$. 
 In the figures \ref{eps1} and \ref{eps2}, we present the behavior of the solutions $(\varphi_\varepsilon^h,\zeta_\varepsilon^h)$ for some decreasing values of $\varepsilon.$ As we can see for $\varepsilon\in [10^{-4},10^{-6}]$ we obtain a good approximations. 
 \begin{figure}[h!]
\centering
 \includegraphics[scale=0.5]{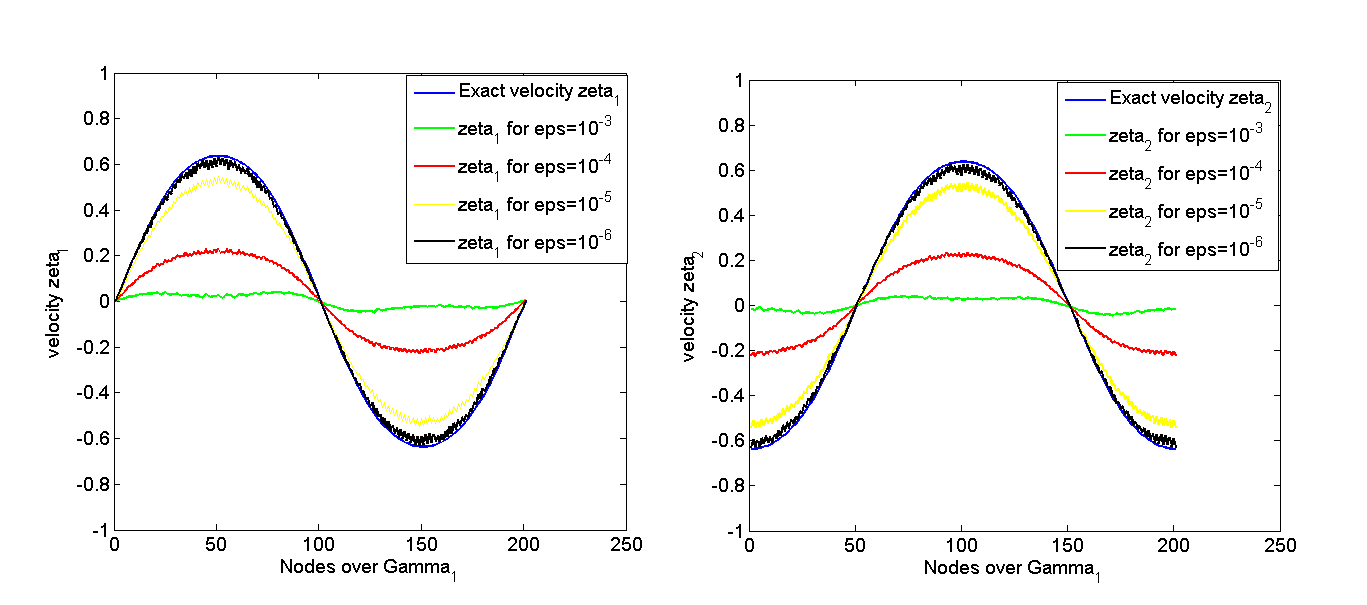}
 \caption{The  solution $\zeta^h_\varepsilon$ for different values of $\varepsilon$}
 \label{eps1}
  \end{figure}
 \begin{figure}[h!]
\centering
 \includegraphics[scale=0.5]{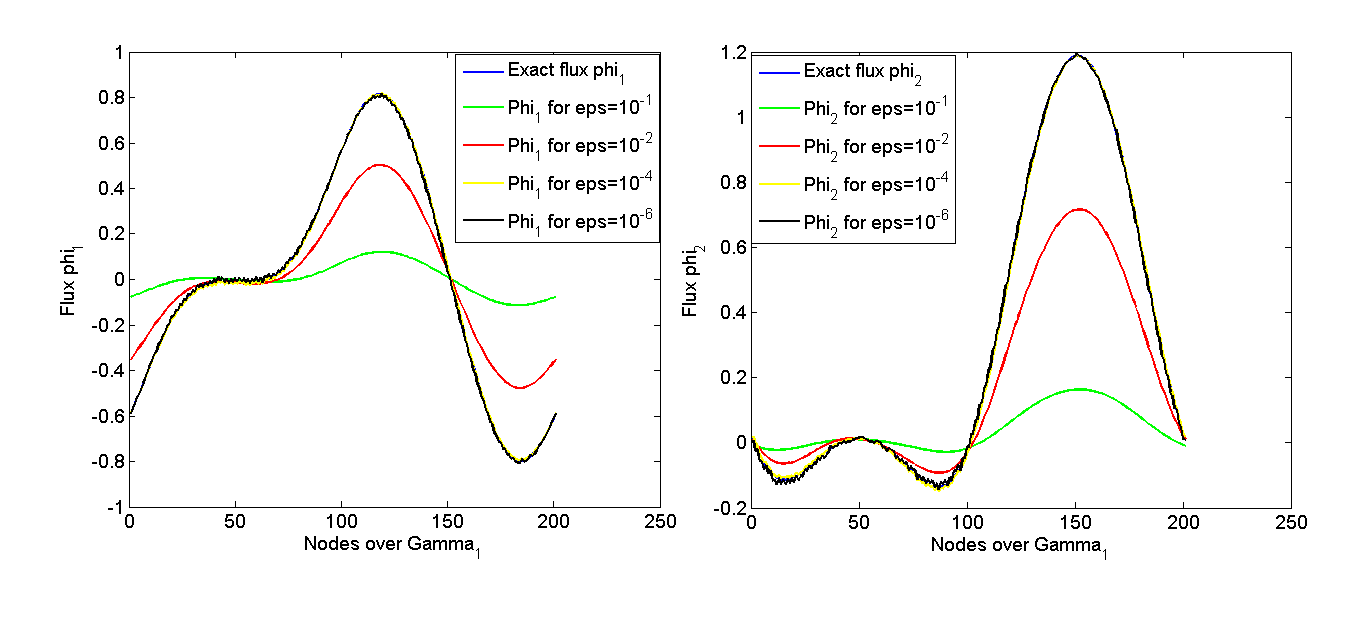}
 \caption{The  solution $\varphi^h_\varepsilon$ for different values of $\varepsilon$}
 \label{eps2}
  \end{figure}
  For $\varepsilon = 10^{-6}$, the approximate problems (\ref{37}) and (\ref{38}) are applied to compute approximate solutions $(\varphi_\varepsilon^h,\zeta_\varepsilon^h)$  of $(\varphi^*, \zeta^*)$ from the Cauchy data $(\psi, \kappa).$ We plot $(\varphi_\varepsilon^h,\zeta_\varepsilon^h)$  in figure   \ref{fig22}  we present the comparison between the exact  and the approximate solutions over the inaccessible boundary $\Gamma$.  As we can see the obtained numerical results are quite satisfactory and  present good approximations of the solutions. 
\begin{figure}[h!]
\centering
 \includegraphics[scale=0.6]{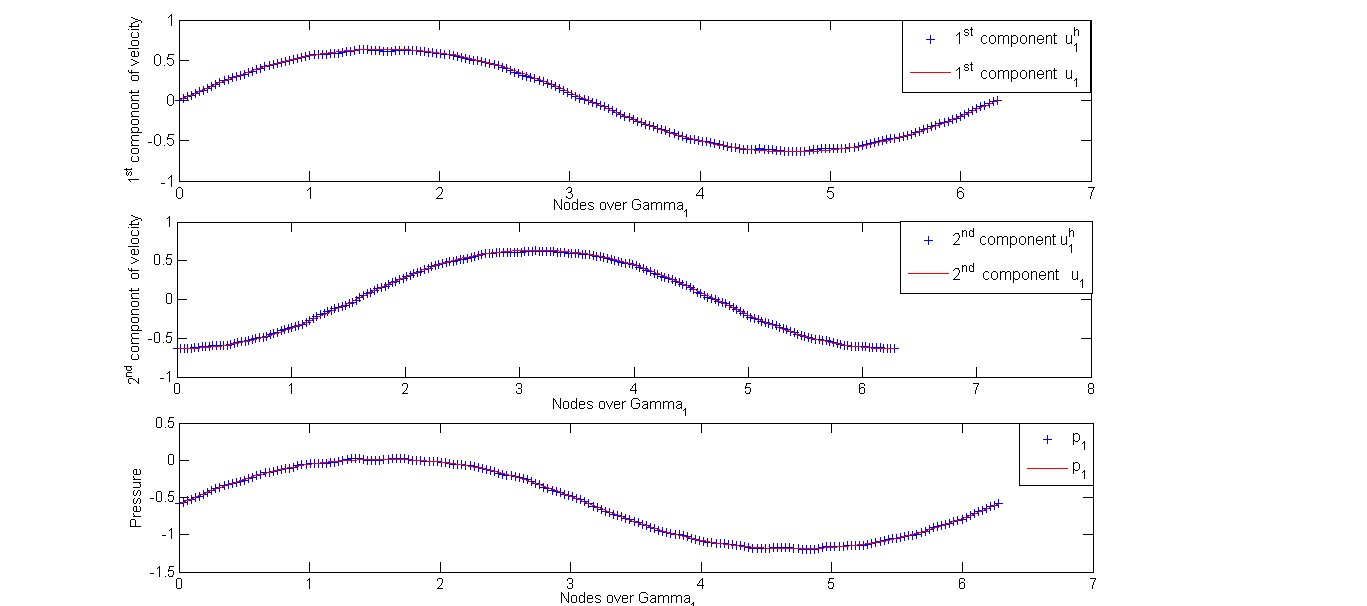}
 \caption{The approximate and exact value of the velocity and pressure on $\Gamma_1$}
\label{fig22}
 \end{figure}
  
 To better evaluate the solution accuracy, we define the $L^2-norm$ relative errors corresponding to solutions $\varphi_\varepsilon^h$, $\zeta^h_\varepsilon,$ and   $(u_1^h,p_1^h)$  as follows:
$$Err_\varphi=\frac{|||\varphi_\varepsilon^h-\varphi^*|||_{0,\Gamma_1}}{\varphi^*}\qquad Err_\zeta=\frac{|||\zeta_\varepsilon^h-\zeta^*|||_{0,\Gamma_1}}{\zeta^*}\qquad Err_u=\frac{|||u_1^h-u^*|||_{0,\Gamma_1}}{u^*}
\quad Err_p=\frac{\|p_1^h-p^*\|_{0,\Gamma_1}}{u^*}
$$   where $(u^*,p^*),$ is the exact solution of $\ref{eq1}$ associated to $(\varphi^*,\zeta^*)$.
 In table \ref{table2}, we examine the numerical convergence of the approximate solution  $(\varphi_\varepsilon^h,\zeta_\varepsilon^h)$ to the exact one $(\varphi_\varepsilon,\zeta_\varepsilon)$ for decreasing value of $h$ (for $\varepsilon=10^{-6}$). We conclude from Table \ref{table2} that the used  approximation method have a good convergence to the  solution $(\varphi,\zeta)$ when $h$ tends to 0.
   \begin{table}[h!]
\centering
\begin{tabular}{|l|c|c|c|c|}
\hline
$h$             & 0.5       & 0.1        & 0.05      &    0.01      \\
\hline
 $  Err_t $   &   0.058    &  0.035        & 0.019 &  0.018 \\
\hline
 $  Err_u  $   &  0.059        &   0.034    & 0.018&  0.017\\
\hline
$Err_p $  &      0.017       &  0.0058    &   0.053 & 0.0015\\
\hline
$Err_\varphi$    &   0.06    &  0.019     &  0.009  &  0.0023\\
\hline
\end{tabular}
\caption{The variation of the errors with respect to $h$ }
\label{table2}
\end{table} 
    

%

Finally, in order to verify the stability of the reconstruction model explored here, a uniformly distributed noise with a noise level $\delta = 1\%, 2\%, 3\%, 4\%$ and $ 5\%$, respectively, is added to $(\psi, \kappa)$  to get $(\psi^\delta, \kappa^\delta):$

$$ \psi_1^{\delta}=[1+\delta\cdot(2\,rand(x)-1)]\psi_1   $$
$$ \psi_2^{\delta}=[1+\delta\cdot(2\,rand(x)-1)]\psi_2    $$
$$  \kappa_1^{\delta}=[1+\delta\cdot(2\,rand(x)-1)]\kappa_1;  $$
$$ \kappa_2^{\delta}=[1+\delta\cdot(2\,rand(x)-1)]\kappa_2;   $$
where $rand(x)$ returns a pseudo-random value drawn from a uniform distribution on $[0, 1]$. The experiments are repeated on the same mesh for $\varepsilon = 10^{-6}$. 
\begin{figure}[h!]
\centering
 \includegraphics[scale=0.4]{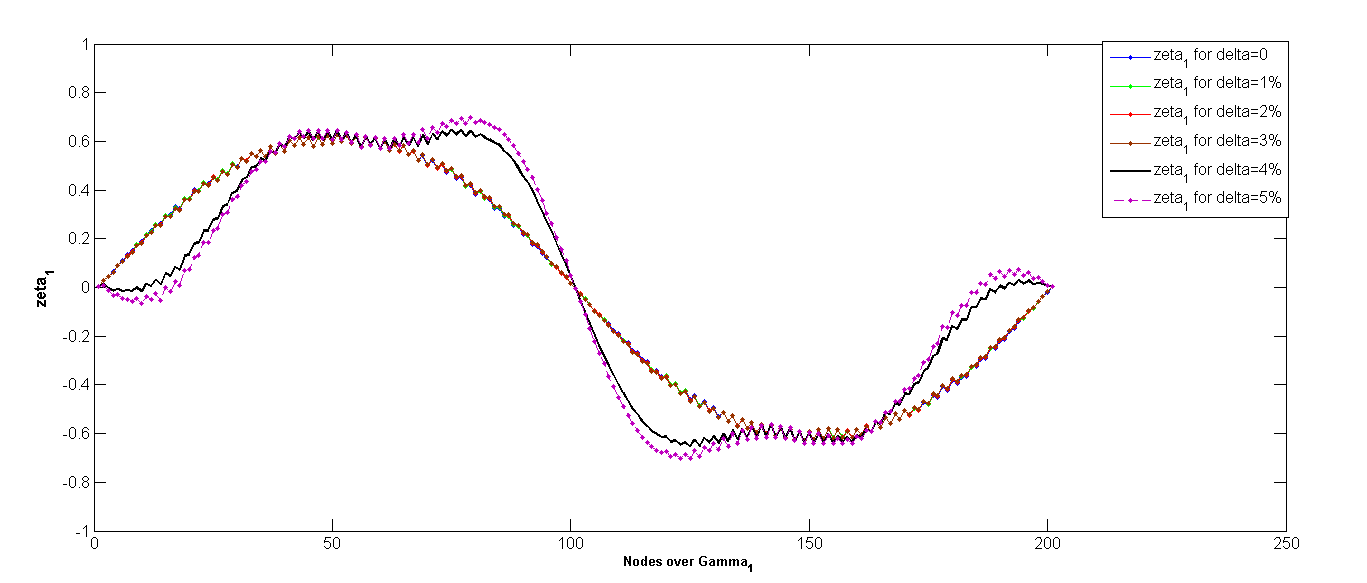}
 \caption{the stability behavior of the solution $\zeta_1$}
 \label{stab1}
 \end{figure}
 \begin{figure}[h!]
\centering
 \includegraphics[scale=0.4]{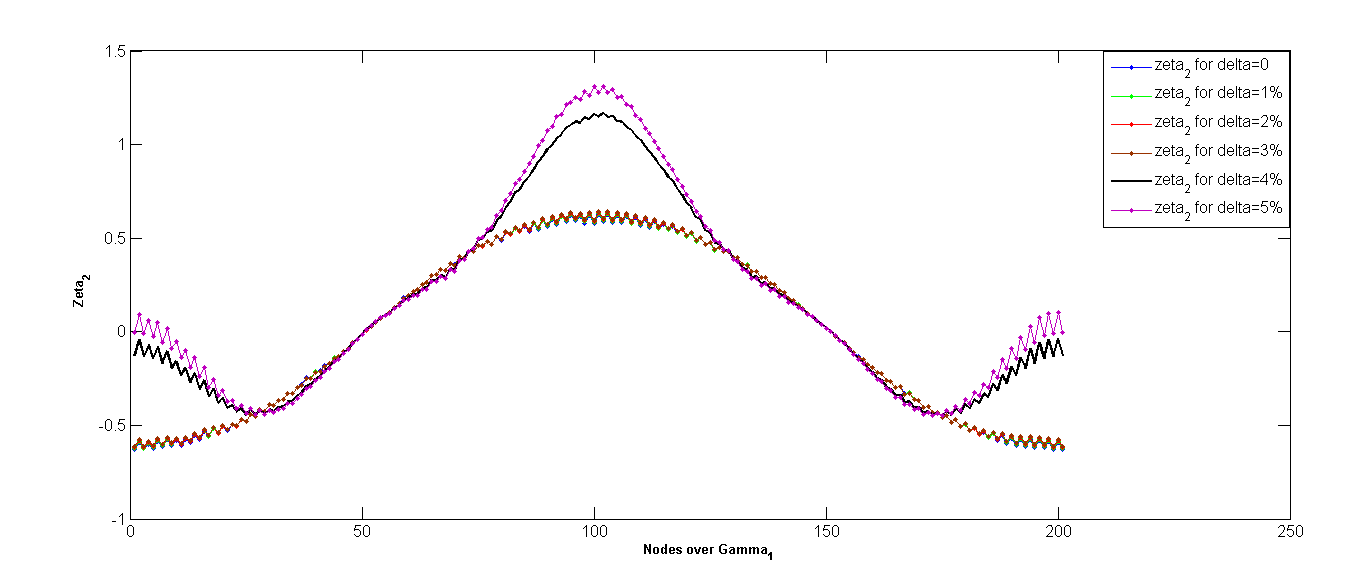}
 \caption{the stability behavior of the solution $\zeta_2$}
 \label{stab2}
 \end{figure}
 \begin{figure}[h!]
\centering
 \includegraphics[scale=0.4]{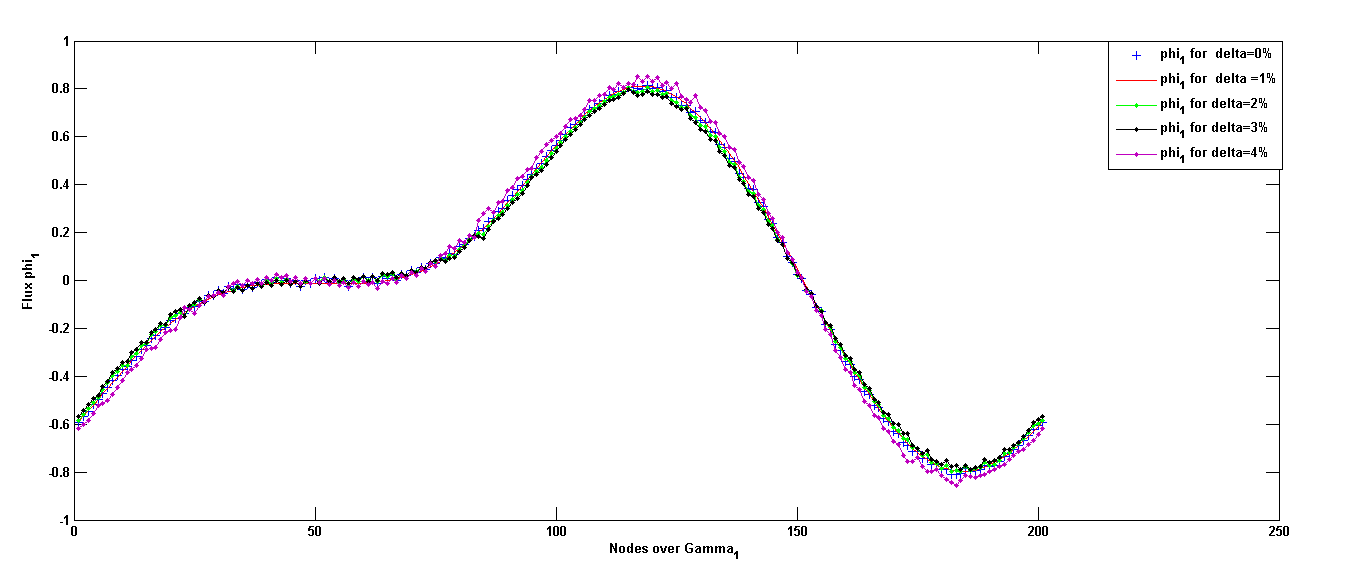}
 \caption{the stability behavior of the solution $\varphi_1$}
 \label{stab3}
 \end{figure}
 \begin{figure}[h!]
\centering
 \includegraphics[scale=0.4]{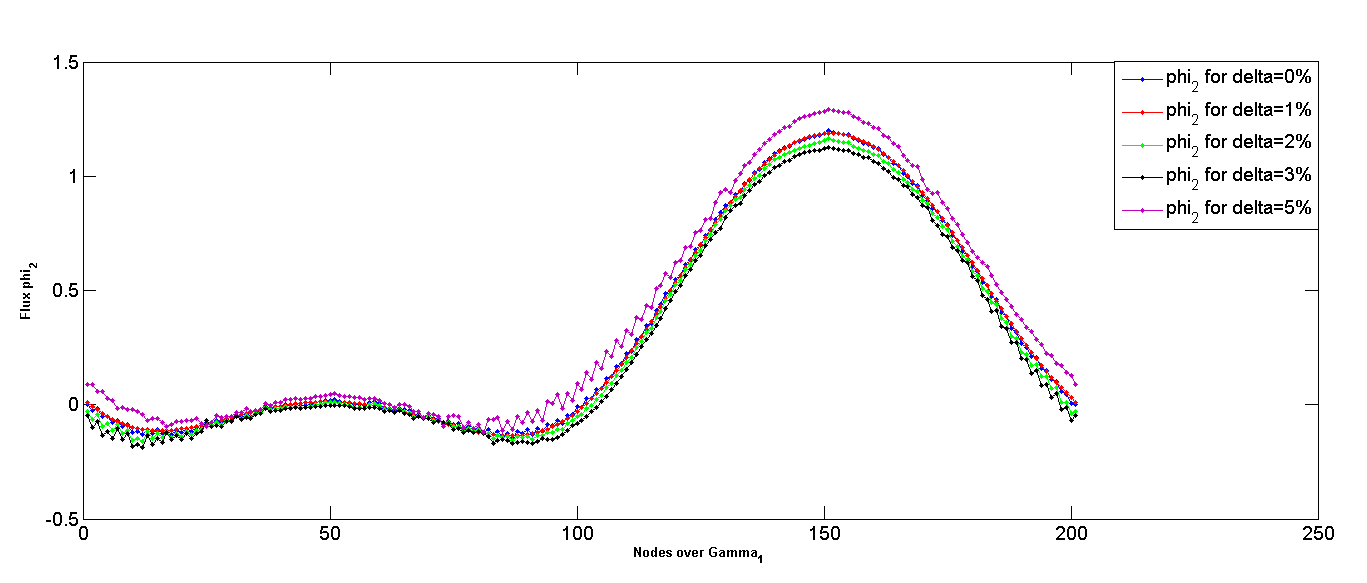}
 \caption{the stability behavior of the solution $\varphi_2$}
 \label{stab4}
 \end{figure}
In Figure \ref{stab1} (resp Figure \ref{stab2}), we present the stability behavior of the solution $\zeta_1$ (resp $\zeta_2$). In Figure \ref{stab3} (resp Figure \ref{stab4}) we present the stability behavior of the solution $\varphi_1$ (resp $\varphi_2$). We conclude from theses figures that, despite the problem is known with its sever instability, the proposed approach produces a convergent and stable numerical solutions with respect to small added noise  into the input data. 

\section{Conclusion}

In this paper, an approach base on a coupled complex boundary method combined with Tikhonov regularization framework  is presented for solving the Cauchy problem  governed by Stokes equation. In this method all boundary conditions are used as parts of a Robin boundary condition. The resulting inverse Cauchy problem is reformulated as minimizing one and performed using Tikhonov regularization approach. Moreover, using the adjoint gradient technic, a simple solver is proposed to compute the regularized solution. Thus, no iteration is needed and the resolution is fast. The obtained theoretical and numerical results show that the proposed approach is feasible, effective and stable.


\bibliographystyle{abbrv}
\bibliography{biblio}
\end{document}